\documentclass[12pt]{amsart}
\usepackage{amsmath}

\newtheorem{theorem}{Theorem}[section]
\newtheorem{lemma}[theorem]{Lemma}
\newtheorem{proposition}[theorem]{Proposition}

 \theoremstyle{definition}
\newtheorem{definition}[theorem]{Definition}

\theoremstyle{remark}

\numberwithin{equation}{section}

\begin{document}

\title[Dirichlet problem]
{The Dirichlet problem for a class of Hessian quotient equations on Riemannian manifolds}

\author{Xiaojuan Chen}
\address{Faculty of Mathematics and Statistics, Hubei Key Laboratory of Applied Mathematics, Hubei University,  Wuhan 430062, P.R. China}
\email{201911110410741@stu.hubu.edu.cn}

\author{Qiang Tu$^{\ast}$}
\address{Faculty of Mathematics and Statistics, Hubei Key Laboratory of Applied Mathematics, Hubei University,  Wuhan 430062, P.R. China}
\email{qiangtu@hubu.edu.cn}

\author{Ni Xiang}
\address{Faculty of Mathematics and Statistics, Hubei Key Laboratory of Applied Mathematics, Hubei University,  Wuhan 430062, P.R. China}
\email{nixiang@hubu.edu.cn}


\thanks{This research was supported by funds from Hubei Provincial Department of Education
Key Projects D20181003, Natural Science Foundation of Hubei Province, China, No. 2020CFB246 and the National Natural Science Foundation of China No. 11971157.}
\thanks{$\ast$ Corresponding author}

\begin{abstract}
In this paper, we consider the Dirichlet problem for a class of Hessian quotient equations on Riemannian manifolds. Under the assumption of an admissible subsolution, we solve the existence and the uniquness for the Dirichlet problem in a domain without any geometric restrictions on the boundary, based on the a priori estimates for the solutions to the Hessian quotient type equations.

\end{abstract}

\maketitle
\section{Introduction}

Let $(M
, g)$ be a compact Riemannian manifold of dimension $n\geq3$ with smooth boundary $\partial M$ and $\overline{M}:=M\cup\partial M$.
In this paper, we study the Dirichlet  problem for a class of Hessian quotient equations
\begin{equation}\label{Eq}
\left\{
\begin{aligned}
&\left(\frac{\sigma_k}{\sigma_l}\right)^{\frac{1}{k-l}}(\lambda[U])=\psi(x,u,\nabla u) &&in~
M,\\
&u = \varphi &&on~\partial M,
\end{aligned}
\right.
\end{equation}
where $U=\tau(\Delta u)g-\nabla^2u$ with $\tau\geq1$, $\nabla^2u$ denotes the Hessian of $u$, $\lambda[U] = (\lambda_1,\cdots, \lambda_n)$ are the
eigenvalues of $U$ with respect to the metric $g$ and
$\psi$ is a positive $C^{\infty}$ function with respect to $(x,z,p)\in \overline{M}\times \mathbb{R}\times T_xM$, where $T_xM$ denotes the tangent space of $M$ at $x$.

Our interest on the solvability of equation \eqref{Eq} is motivated from
the complex Monge-Amp\`ere type equations. Recently, Harvey-Lawson \cite{Ha12, Ha11-} introduced a class of functions $u\in C^2(\mathbb{C}^n)$, named $(n-1)$-plurisubharmonic, such that the complex Hessian matrix
\begin{eqnarray}\label{c-1}
\bigg[\Big(\sum_{i=1}^{n}\frac{\partial^2u}{\partial
z_m\partial\overline{z}_m}\Big)\delta_{ij}-\frac{\partial^2u}{\partial
z_i\partial\overline{z}_j}\bigg]_{1\leq i,j\leq n}
\end{eqnarray}
is nonnegative definite. For $(n-1)$-plurisubharmonic functions, one can consider the following complex Monge-Amp\`ere equations
\begin{eqnarray}\label{c}
\mathrm{det}\bigg(\Big(\sum_{i=1}^{n}\frac{\partial^2u}{\partial
z_m\partial\overline{z}_m}\Big)\delta_{ij}-\frac{\partial^2u}{\partial
z_i\partial\overline{z}_j}\bigg)=\psi.
\end{eqnarray}
If $\psi$ does not depend on $\nabla u$,
the Dirichlet problem for \eqref{c} on strict pseudo-convex domains in $\mathbb{C}^n$  was solved by Li \cite{Li04}, who also considered a general class of operators. Tosatti-Weinkove \cite{To17, To19} showed that the associated complex
Monge-Amp\`ere equation can be solved on any compact K\"{a}hler manifold.
Harvey-Lawson \cite{Ha11, Ha12} investigated the corresponding complex Monge-Amp\`ere equation and sloved the Dirichlet problem with $\psi=0$ on suitable domains. Then Han-Ma-Wu \cite{Ha09}
considered $k$-convex solutions of complex Laplace equation.
Moreover, the complex Hessian equation involving a gradient term on the left hand sides has attracted the interest of many authors due to its geometric applications such as the Gauduchon conjecture, which was solved by Sz\'ekelyhidi-Tosatti-Weinkove in their work \cite{Sz17}, and also see Guan-Nie \cite{Guan21}.
 For more references, we refer the readers to \cite{Ga84,Fu10, Fu15, Sz18} and
references therein.

If the complex Hessian matrix is replaced by the real Hessian matrix in \eqref{c-1},
a natural question is whether we can study the regularity and solvability to
the Dirichlet boundary problem for this kind of  fully nonlinear  equation (such as \eqref{Eq}). This work is a further study on the Dirichlet problem for \eqref{Eq} with gradient terms on the right sides of the equation following a recent work by Chu-Jiao \cite{Chu20}.
 To ensure the ellipticity of \eqref{Eq}, we need $\lambda[U]\in \Gamma_k$. Hence we introduce the following definition.

\begin{definition}\label{def-1}
A function $u\in C^2(M)$ is called admissible (i.e.$(\eta, k)$-convex) if $\lambda[U] \in \Gamma_k$ for any $x\in M$,
 where $\Gamma_k$ is the Garding cone
\begin{eqnarray*}\label{cone}
\Gamma_{k}=\{\lambda \in \mathbb{R} ^n: \sigma_{j}(\lambda)>0, \forall ~ 1\leq j \leq k\}.
\end{eqnarray*}
\end{definition}

 The main theorem is as follows.
\begin{theorem}\label{main}
Let $l+2\leq k\leq n$, $\varphi \in C^{\infty}(\partial M)$,  $\psi\in C^{\infty}(\overline{M}\times \mathbb{R}\times T_xM)$ with $\psi, \psi_z>0$.
Assume that there exists an admissible subsolution $\underline{u}\in C^2(\overline{M})$ satisfying
\begin{equation}\label{Eq-sub}
\left\{
\begin{aligned}
&\left(\frac{\sigma_k}{\sigma_l}\right)^{\frac{1}{k-l}}\left(\lambda[\underline{U}]\right)\geq \psi(x,\underline{u},\nabla\underline{u}) &&in~
M,\\
&\underline{u} = \varphi &&on~\partial M,
\end{aligned}
\right.
\end{equation}
where $\underline{U}=\tau(\Delta \underline{u})g-\nabla^2\underline{u}$. Then the Dirichlet problem \eqref{Eq} is uniquely solvable for $u\in C^{\infty}(\overline{M})$ with $\lambda[U] \in\Gamma_k$.
\end{theorem}
In order to prove Theorem \ref{main}, a major challenge comes from second order estimates for a domain with arbitrary boundary shape except being smooth. We establish the following global second order estimates.
\begin{theorem}\label{main-1}
Let $l+2\leq k\leq n$, $\varphi \in C^{2}(\partial M)$,  $\psi\in C^{2}(\overline{M}\times \mathbb{R}\times T_xM)$ with $\psi, \psi_z>0$, $u\in C^4(M)\cap C^2(\overline{M})$ be an admissible solution of Dirichlet problem \eqref{Eq}.
Assume that there exists an admissible subsolution $\underline{u}\in C^2(\overline{M})$ satisfying \eqref{Eq-sub}. There exists $C$ depending on $n, k, l, \|u\|_{C^1}, \|\underline{u}\|_{C^2}, \inf \psi,  \|\psi\|_{C^2}$ and the curvature
tensor $R$ such that
$$\sup_{ \overline{M}} |\nabla^2 u | \leq C.$$
\end{theorem}

If $U=\tau(\Delta u)g-\nabla^2u$ is replaced by the Hessian matrix $\nabla^2u$, equation \eqref{Eq} becomes the classical Hessian quotient equation
\begin{equation}\label{hq-Eq-1}
\left(\frac{\sigma_k}{\sigma_l}\right)^{\frac{1}{k-l}}\left(\lambda[\nabla^2 u]\right)=\psi(x,u,\nabla u) \quad \mbox{in}~
M,
\end{equation}
which has been widely studied in the past decades for the Euclidean case.
When $\psi=\psi(x)$,
 $C^2$ estimates  were treated by Caffarelli-Nirenberg-Spruck \cite{CNS85} for $l=0$,  where they treated a general class of fully nonlinear equations under conditions on the geometry of $\partial M$, followed by \cite{Guan94,LY90}. Then such estimates for \eqref{hq-Eq-1} have been established  by Trudinger \cite{Tr95}, Ivochkina-Trudinger-Wang \cite{ITW04} who considered the degenerate case, Guan \cite{Guan14} who considered to treat a general class of fully nonlinear equations on Riemannian
manifolds, without geometric restrictions to the boundary.
 When $\psi=\psi(x, u, \nabla u)$, equation \eqref{hq-Eq-1} falls into the setup of Guan-Jiao \cite{Guan15} (see also \cite{Guan99}), and the $C^2$ estimate was obtained under the concavity assumption of $\psi$ on $\nabla u$.  In Theorem \ref{main-1},  we remove this concavity assumption for equation \eqref{Eq}.

It would be worthwhile to note that this type of equation \eqref{Eq} arise naturally from many other important geometric problems. Another example is a class of prescribed curvature problems. A $(0, 2)$-tensor on  a hypersurface $M\subset \mathbb{R}^{n+1}$ is defined by
\begin{eqnarray*}
\eta_{ij}=Hg_{ij}-h_{ij},
\end{eqnarray*}
where $g_{ij}$ is the induced metric of $M$ from $\mathbb{R}^{n+1}$,
$h_{ij}$ and $H$ are the second fundamental form and the mean
curvature of $M$ respectively.  The $(n-1)$-convex
hypersurface (i.e. $\eta_{ij}$ is nonnegative definite) has
been studied intensively by Sha \cite{S86, S87}, Wu \cite{Wu87}, and
Harvey-Lawson \cite{Ha13}. Recently, Chu-Jiao \cite{Chu20}
considered the following prescribed curvature problem
\begin{eqnarray*}
\sigma_{k}(\eta_{ij}(X))=\psi(X, \nu(X)), \quad X \in M,
\end{eqnarray*}
where $\nu$ was the unit outer normal vector of $M$. Later on,
The authors \cite{Chen20} studied the corresponding Hessian
quotient type  prescribed curvature problem. Moreover,
an analogue of equation \eqref{Eq}  on compact  manifolds also appeared naturally in conformal geometry, see Gursky-Viaclovsky \cite{Ger03}, Li-Sheng \cite{LS11} and Sheng-Zhang \cite{Sheng07}.

The organization of the paper is as follows. In Section 2 we start with some preliminaries.
Our proof of the estimates heavily depends on
results in Section 3 and 4. $C^1$ estimates are given in Section 3. In Section 4 we derive the global estimates for second derivatives, and finish the proof of Theorem \ref{main} and Theorem \ref{main-1}.

\section{Preliminaries}

Let $\lambda=(\lambda_1,\dots,\lambda_n)\in\mathbb{R}^n$, we recall
the definition of elementary symmetric function for $1\leq k\leq n$
\begin{equation*}
\sigma_k(\lambda)= \sum _{1 \le i_1 < i_2 <\cdots<i_k\leq
n}\lambda_{i_1}\lambda_{i_2}\cdots\lambda_{i_k}.
\end{equation*}
We also set $\sigma_0=1$ and $\sigma_k=0$ for $k>n$ or $k<0$. The Garding cone is defined by
\begin{equation*}
\Gamma_k  = \{ \lambda  \in \mathbb{R}^n :\sigma _i (\lambda ) >
0,\forall 1 \le i \le k\}.
\end{equation*}
We denote $\sigma_{k-1}(\lambda|i)=\frac{\partial
\sigma_k}{\partial \lambda_i}$ and
$\sigma_{k-2}(\lambda|ij)=\frac{\partial^2 \sigma_k}{\partial
\lambda_i\partial \lambda_j}$. Next, we list some properties of
$\sigma_k$ which will be used later.

\begin{proposition}\label{sigma}
Let $\lambda=(\lambda_1,\dots,\lambda_n)\in\mathbb{R}^n$ and $1\leq
k\leq n$, then we have

(1) $\Gamma_1\supset \Gamma_2\supset \cdot\cdot\cdot\supset
\Gamma_n$;

(2) $\sigma_{k-1}(\lambda|i)>0$ for $\lambda \in \Gamma_k$ and
$1\leq i\leq n$;

(3) $\sigma_k(\lambda)=\sigma_k(\lambda|i)
+\lambda_i\sigma_{k-1}(\lambda|i)$ for $1\leq i\leq n$;

(4)
$\sum_{i=1}^{n}\frac{\partial[\frac{\sigma_{k}}{\sigma_{l}}]^{\frac{1}{k-l}}}
{\partial \lambda_i}\geq [\frac{C^k_n}{C^l_n}]^{\frac{1}{k-l}}$ for
$\lambda \in \Gamma_{k}$ and $0\leq l<k$;

(5) $\Big[\frac{\sigma_k}{\sigma_l}\Big]^{\frac{1}{k-l}}$ are
concave in $\Gamma_k$ for $0\leq l<k$;

(6) If $\lambda_1\geq \lambda_2\geq \cdot\cdot\cdot\geq \lambda_n$,
then $\sigma_{k-1}(\lambda|1)\leq \sigma_{k-1}(\lambda|2)\leq
\cdot\cdot\cdot\leq \sigma_{k-1}(\lambda|n)$ for $\lambda \in
\Gamma_k$;

(7)
$\sum_{i=1}^{n}\sigma_{k-1}(\lambda|i)=(n-k+1)\sigma_{k-1}(\lambda)$.
\end{proposition}

\begin{proof}
All the properties are well known. For example, see Chapter XV in
\cite{Li96} or \cite{Hui99} for proofs of (1), (2), (3),  (6) and
(7); see Lemma 2.2.19 in \cite{Ger06} for the proof of (4); see
\cite{CNS85} and \cite{Li96} for the proof of (5).
\end{proof}

The generalized Newton-MacLaurin inequality is as follows, which
will be used later.
\begin{proposition}\label{NM}
For $\lambda \in \Gamma_m$ and $m > l \geq 0$, $ r > s \geq 0$, $m
\geq r$, $l \geq s$, we have
\begin{align}
\Bigg[\frac{{\sigma _m (\lambda )}/{C_n^m }}{{\sigma _l (\lambda
)}/{C_n^l }}\Bigg]^{\frac{1}{m-l}} \le \Bigg[\frac{{\sigma _r
(\lambda )}/{C_n^r }}{{\sigma _s (\lambda )}/{C_n^s
}}\Bigg]^{\frac{1}{r-s}}. \notag
\end{align}
\end{proposition}
\begin{proof}
See \cite{S05}.
\end{proof}

In this paper, $\nabla$ denotes the Levi-Civita connection on $(M , g)$ and the curvature tensor
is defined by
$$R(X, Y )Z = - \nabla_X \nabla_Y Z + \nabla_Y \nabla_X Z + \nabla_{[X,Y]}Z.$$
Let $e_1,e_2,\cdots,e_n$ be local frames on $M$ and denote $g_{ij}=g(e_i,e_j)$, $\{g^{ij}\}=\{g_{ij}\}^{-1}$,
while the Christoffel symbols $\Gamma^k_{ij}$ and curvature coefficients are given respectively by $\nabla_{e_i}e_j=\Gamma^k_{ij}e_k$ and
$$R_{ijkl}=g(R(e_k,e_l)e_j,e_i),\quad R^i_{jkl}=g^{im}R_{mjkl}.$$
We shall write $\nabla_i=\nabla_{e_i}$, $\nabla_{ij}=\nabla_i\nabla_j-\Gamma^k_{ij}\nabla_k$, etc.
For a differentiable function $u$ defined on $M$, we usually identify $\nabla u$ with its gradient,
and use $\nabla^2 u$ to denote its Hessian which is locally given by $\nabla_{ij} u= \nabla_i(\nabla_j u)
-\Gamma^k_{ij}\nabla_k u$. We note that $\nabla_{ij} u=\nabla_{ji} u$ and
\begin{equation}\label{req1}
  \nabla_{ijk} u-\nabla_{jik} u=R^l_{kij}\nabla_lu,
\end{equation}
\begin{equation}\label{req0}
  \nabla_{ij}(\nabla_ku) = \nabla_{ijk}u + \Gamma^l_{ik}\nabla_{jl}u +\Gamma^l_{jk}\nabla_{il}u + \nabla_{\nabla_{ij}e_k}u,
\end{equation}
\begin{equation}\label{req2}
  \nabla_{ijkl}u-\nabla_{ikjl}u=R^m_{ljk}\nabla_{im}u+\nabla_iR^m_{ljk}\nabla_mu,
\end{equation}
\begin{equation}\label{req3}
  \nabla_{ijkl}u-\nabla_{jikl}u=R^m_{kij}\nabla_{ml}u+R^m_{lij}\nabla_{km}u.
\end{equation}
From \eqref{req2} and \eqref{req3}, we obtain
\begin{eqnarray}\label{req4}
 \nonumber \nabla_{ijkl}u-\nabla_{klij}u&=& R^m_{ljk}\nabla_{im}u+\nabla_iR^m_{ljk}\nabla_mu+R^m_{lik}\nabla_{jm}u \\
   && +R^m_{jik}\nabla_{lm}u+R^m_{jil}\nabla_{km}u+\nabla_kR^m_{jil}\nabla_mu.
\end{eqnarray}
For convenience, we introduce
the following notations
$$F(U)=\bigg[\frac{\sigma_k(\lambda[U])}{\sigma_l(\lambda[U])}\bigg]^{\frac{1}{k-l}},
\quad F^{ij}=\frac{\partial F}{\partial U_{ij}}, \quad F^{ij, r
s}=\frac{\partial^2 F}{\partial U_{ij}\partial U_{rs}},\quad
  Q^{ij}=\frac{\partial F}{\partial u_{ij}}, \quad Q^{ij, r
s}=\frac{\partial^2 F}{\partial u_{ij}\partial u_{rs}}.$$
Let $u\in C^{\infty}(\overline{M})$ be an admissible solution of equation \eqref{Eq}. Under orthonormal local frames $e_1,\cdots,e_n$, equation \eqref{Eq} is expressed in the form
\begin{equation}\label{FU}
  F(U):=f(\lambda[U])=\psi.
\end{equation}
For simplicity, we shall still write equation \eqref{Eq} in the form \eqref{FU} even if $e_1,\cdots,e_n$ are not necessarily orthonormal, although more precisely it should be
$$F([\gamma^{ik}U_{kl}\gamma^{lj}])=\psi,$$
where $\gamma^{ij}$ is the square root of $g^{ij}: \gamma^{ik}\gamma^{kj}=g^{ij}$. Whenever we differentiate the equation, it will make no difference as long as we use covariant derivatives.
Assume that $\overline{A}$ is an $n\times n$ matrix and $T:\overline{A}\rightarrow T(\overline{A})$ is defined as $T(\overline{A})=\tau (tr(\overline{A}))I-\overline{A}$. Let $Q=F\small{\circ}T$, then equation \eqref{Eq}
can also be written as
\begin{equation*}
  Q(\nabla^2u):=\widetilde{f}(\widetilde{\lambda}[\nabla^2u])=\psi,
\end{equation*}
Hence $Q^{ij}=\frac{\partial Q}{\partial u_{ij}}=\frac{\partial F}{\partial u_{ij}}=\tau\sum_lF^{ll}\delta_{ij}-F^{ij}$ and then
\begin{equation}\label{Quii}
  Q^{ij}u_{ij}=F^{ij}U_{ij}=f_i\lambda_i=\widetilde{f}_i\widetilde{\lambda}_i=\psi.
\end{equation}
Differentiating \eqref{FU}, we get
\begin{equation}\label{Quiik}
  Q^{ij}\nabla_ku_{ij}=F^{ij}\nabla_kU_{ij}=\psi_k+\psi_zu_k+\psi_{p_i}u_{ik}.
\end{equation}


The following propositions are essential which will be used later.  More details can be seen in \cite{Chen20}.
\begin{proposition}\label{ellipticconcave}
Let $M$ be a smooth $(\eta, k)$-convex closed hypersurface in $\mathbb{R}^{n+1}$
and $0\leq l< k-1$. Then the operator
\begin{equation*}
F(U)=\left(\frac{\sigma_k(\lambda[U])}{\sigma_{l}(\lambda[U])}\right)^{\frac{1}{k-l}}
\end{equation*}
is elliptic and concave with respect to $U$. Moreover we have
\begin{equation*}
\sum F^{ii} \geq \left(\frac{C_n^k}{C_n^l}\right)^{\frac{1}{k-l}}.
\end{equation*}
\end{proposition}

\begin{proposition}\label{th-lem-07}
Let $U$ be a diagonal matrix with $\lambda[U]\in \Gamma_k$, $0\leq l \leq k-2$ and $k\geq 3$. Then
\begin{equation*}
-F^{1i, i1}(U)=\frac{F^{11}-F^{ii}}{U_{ii}-U_{11}},\quad \forall~i\geq2.
\end{equation*}
\end{proposition}
\section{$C^1$ Estimates}

In this section, we consider the lower and upper bounds, gradient estimates for the admissible solution to equation \eqref{Eq}.

\begin{lemma}\label{C0}
Let $u\in C^{\infty}(\overline{M})$ be an admissible solution for equation \eqref{Eq}.
Under the assumptions mentioned  in Theorem \ref{main},  then there exists a positive constant $C$ depending only on $ \sup_{\partial M}\varphi$ and the subsolution $\underline{u}$ such that
 $$\sup_{x \in \overline{M}} |u(x)|\leq C.$$
 \end{lemma}

\begin{proof}
On the one hand, according to  Definition \ref{def-1}, it is easy to see that
 $\lambda[U]\in\Gamma_k\subset\Gamma_1$, which implies that $tr(\lambda[U])=(\tau n-1)\Delta u>0$.
Combined with the maximum principle, we have
$$\sup_{\overline{M}}u\leq\sup_{\partial M}\varphi.$$
On the other hand, we know that  there exists an admissible subsolution $\underline{u}\in C^2(\overline{M})$ satisfying \eqref{Eq-sub}.
By the fact $\psi_z>0$ and the comparison principle,
$$u\geq\underline{u}, \quad \forall~x \in\overline{M}.$$
\end{proof}

\begin{lemma}\label{C1-0}
Let  $l+2\leq k\leq n$, $\varphi \in C^{\infty}(\partial M)$,  $\psi\in C^{\infty}(\overline{M}\times \mathbb{R}\times T_xM)$ with $\psi, \psi_z>0$. If $u\in C^2(\overline{M})$ is the solution of equation \eqref{Eq},
then
 $$\sup_{M}|\nabla u|\leq C(1+\sup_{\partial M}|\nabla u|),$$
where $C$ is a constant depending on $n,k,l,\|u\|_{C^0},\|\psi\|_{C^1}$ and the curvature tensor $R$.
\end{lemma}
\begin{proof}
Consider the auxiliary function
$$P(x)=ve^{\phi(u)},$$
where $v=1+\frac{1}{2}|\nabla u|^2$, $\phi(u): \mathbb{R}\longrightarrow \mathbb{R}$ is a function satisfying
$$\phi'(u)>0, \quad \phi''(u)-(\phi'(u))^2\geq
\varepsilon$$ for some positive constant $\varepsilon$ depending on $\|u\|_{C^0}$.

Suppose that $P$ attains its maximum at $x_0\in M$. By rotating the coordinates, we diagonal the matrix $\nabla^2u$. In the following, we write simply $u_i=\nabla_iu$, $u_{ij}=\nabla_{ij}u$ and $u_{ijk}=\nabla_ku_{ij}$, then at $x_0$,
\begin{equation}\label{Pi}
  0=P_i=(u_{ii}u_i+v\phi'u_i)e^{\phi(u)},
\end{equation}
and
\begin{eqnarray}\label{Pii}
  0\geq P_{ii}=\left(u_{ii}^2+u_ku_{kii}+2u_i^2u_{ii}\phi'+u_i^2v\left(\phi''+(\phi')^2\right)+v\phi'u_{ii}\right)
  e^{\phi (u)}.
\end{eqnarray}
 We assume that $v\leq |\nabla u|^2$, i.e., $|\nabla u|^2\geq 2$. Otherwise our result holds.
Let
$$\mathcal{S}=\{i\in (1, \cdots, n) \mid u_i \neq 0\}.$$
Obviously $\mathcal{S}\neq \emptyset$ and   we derive
$$u_{ii}=-v\phi^{\prime}<0 , \quad i \in \mathcal{S}$$
by \eqref{Pi}. From the mean curvature $H > 0$, we have
$$Q^{ii}\geq\sum_{l\neq i} F^{ll}\geq \frac{1}{2} \sum_{l} F^{ll},$$
which implies
\begin{eqnarray}\label{Qii}
\nonumber Q^{ii}u_i^2&=&\sum_{i\in S} Q^{ii}u_i^2\geq \sum_{i\in S} \left(\frac{1}{2} \sum_{l} F^{ll}\right)u_i^2\\
 &=& \left(\frac{1}{2} \sum_{l} F^{ll}\right) |\nabla u|^2=\frac{1}{2(\tau n-1)}\left( \sum_{l} Q^{ll}\right) |\nabla u|^2.
\end{eqnarray}
 Since $Q^{ii}\geq0$ and by Ricci identity, we have $u_{kij}=u_{ijk}+R^l_{jki}u_l$, then
\begin{eqnarray}\label{c1eq}
 \nonumber 0&\geq& Q^{ii}\left(u_ku_{kii}+2u_i^2u_{ii}\phi'+u_i^2v\left(\phi''+(\phi')^2\right)+v\phi'u_{ii}\right)\\
   \nonumber &=& \psi_ku_k+\psi_zu_k^2+\psi_{p_k}u_ku_{kk}+R^l_{iki}Q^{ii}u_ku_l+2\phi'Q^{ii}u_i^2u_{ii}\\
   \nonumber&&+v\left(\phi''+(\phi')^2\right)Q^{ii}u_{i}^2
    +v\phi'Q^{ii}u_{ii}\\
    \nonumber&\geq& \psi_ku_k-v\phi'\psi_{p_k}u_k+R^l_{iki}Q^{ii}u_ku_l+v\left(\phi''-(\phi')^2\right)Q^{ii}u_{i}^2+v\phi'\psi\\
    &\geq&\sum_lQ^{ll}\left(\frac{\varepsilon}{4(\tau n-1)}|\nabla u|^4-\overline{C}|\nabla u|^2\right)-
    \overline{C}\phi'|\nabla u|^3-\overline{C}\phi'|\nabla u|^2-\overline{C}|\nabla u|,
\end{eqnarray}
where $\overline{C}$ is a constant depending on $\|\psi\|_{C^1}$ and the curvature tensor $R$.

On the one hand, if $\frac{\varepsilon}{4(\tau n-1)}|\nabla u|^4-\overline{C}|\nabla u|^2\leq0$, then $|\nabla u|\leq C$.  Otherwise by \eqref{c1eq} and the fact
$\sum_lQ^{ll}\geq(\tau n-1)\left(\frac{C_n^k}{C_n^l}\right)^{\frac{1}{k-l}}$, we derive
 $$(\tau n-1)\left(\frac{C_n^k}{C_n^l}\right)^{\frac{1}{k-l}}\left(\frac{\varepsilon}{4(\tau n-1)}|\nabla u|^4-\overline{C}|\nabla u|^2\right)-
    \overline{C}\phi'|\nabla u|^3-\overline{C}\phi'|\nabla u|^2-\overline{C}|\nabla u|\leq0,$$
 then $|\nabla u|\leq C$, the lemma is proved.
\end{proof}

Next, we derive the global $C^1$ estimates for the solution of equation \eqref{Eq}.

\begin{theorem}\label{C1}
Let $u\in C^{\infty}(\overline{M})$ be an admissible solution for equation \eqref{Eq}.
Under the assumptions mentioned  in Theorem \ref{main}, then
 $$\sup_{\overline M}|\nabla u|\leq C,$$
where $C$ is a constant depending on $n, k, l$, $\|u\|_{C^0}$, $\|\underline{u}\|_{C^1}$, $\|\varphi\|_{C^1}$, $\|\psi\|_{C^1}$ and the curvature tensor $R$.
\end{theorem}

\begin{proof}

From Lemma \ref{C1-0}, we are left with the task of estimating the exterior normal derivative of $u$ on $\partial M$. Let $h$ be the harmonic function in $M$ which equals $\varphi$ on $\partial M$, then we derive
\begin{equation*}
\left\{
\begin{aligned}
&\Delta (u-h)>0 &&in~
M,\\
&u-h=0 &&on~\partial M.
\end{aligned}
\right.
\end{equation*}
The maximum principle implies $u\leq h$ in $M$. Therefore,
$$\underline{u}\leq u\leq h \quad \mbox{in}~M.$$
Since they are all equal to $\varphi$ on $\partial M$, then
$$\nabla_{\nu}h\leq\nabla_{\nu}u\leq\nabla_{\nu}\underline{u} \quad \mbox{on}~\partial M,$$
where $\nu$ is the exterior normal derivative of $u$ on $\partial M$.
Thus, we have
$$\sup_{\partial M}|\nabla u|\leq C,$$
which completes the proof.
\end{proof}

\section{Global Estimates for second derivatives}
In this section, we prove the global  second order estimates and give the proof of Theorem \ref{main} and \ref{main-1}. Firstly, we need to derive the following theorem.

\begin{theorem}\label{C2-0}
Let $u\in C^{\infty}(M)$ be an admissible solution for equation \eqref{Eq}.  Then
there exists a constant $C$ depending only on $n, k, l, \|u\|_{C^1}, \|\underline{u}\|_{C^2}, \|\psi\|_{C^2}$ and the curvature tensor
$R$
 such that
$$\sup_{M} |\nabla^2 u | \leq C(1+\sup_{\partial M}|\nabla^2 u |).$$
\end{theorem}

\begin{proof}
Taking the auxiliary function
\begin{equation*}
\widehat{H}=\log \widetilde{\lambda}_{\mbox{max}}(\nabla^2u)+\frac{a}{2}|\nabla u|^2+A(\underline{u}-u),
\end{equation*}
where $\widetilde{\lambda}_{\mbox{max}}(\nabla^2u)$ is the largest eigenvalue of $\nabla^2u$, $a\leq1$ and $A\geq1$ are constants to be determined later, $x_0$ is the maximum point of $\widehat{H}$. We choose a local orthonormal frame  $\{e_{1}, e_{2}, \cdots, e_{n}\}$ near $x_0$ such that $\nabla_{e_i}e_j=0$, i.e. $\Gamma_{ij}^k=0$ at $x_0$ for any $1\leq i,j,k\leq n$.
For convenience, we write $u_i=\nabla_iu, u_{ij}=\nabla_{ij}u, u_{ijl}=\nabla_lu_{ij}$ , $u_{ijrs}=\nabla_{rs}u_{ij}$ and $R^m_{ijs;l}=\nabla_lR^m_{ijs}$. Assume that
$$u_{11}\geq u_{22}\geq \cdots \geq u_{nn}$$
at $x_0$. Recalling that $U_{ii}=\tau\Delta u-u_{ii}$, we have
$$U_{11}\leq U_{22}\leq\cdots\leq U_{nn}.$$
It can follows that
$$
F^{11}\geq F^{22}\geq\cdots\geq F^{nn}\quad \mbox{and} \quad Q^{11}\leq Q^{22}\leq\dots\leq Q^{nn}.$$
We define a new function $\widetilde{H}$ by
\begin{equation*}
\widetilde{H}=\log u_{11}+\frac{a}{2}|\nabla u|^2+A(\underline{u}-u).
\end{equation*}
Then at $x_0$, we have
\begin{equation}\label{Hi}
  0=\widetilde{H}_i=\frac{u_{11i}}{u_{11}}+au_iu_{ii}+A(\underline{u}-u)_i,
\end{equation}
\begin{equation}\label{Hii}
  0\geq \widetilde{H}_{ii}=\frac{u_{11ii}u_{11}-u_{11i}^2}{u_{11}^2}+au_{ii}^2+au_ku_{kii}+A(\underline{u}-u)_{ii}.
\end{equation}
We divide our proof into four steps.

\textbf{Step 1}:  We show that
\begin{eqnarray}\label{ht-c2-1}
\nonumber0&\geq& - \frac{2}{u_{11}} \sum_{i\geq 2} Q^{1i, i1} u_{1i1}^2 -\frac{Q^{ii}u_{11i}^2}{u_{11}^2}
+\frac{aQ^{ii}u_{ii}^2}{2}+AQ^{ii}(\underline{u}-u)_{ii}-C_0\sum_i Q^{ii}\\
&&-\frac{C_0^2\sum_iQ^{ii}}{2au_{11}^2}-\frac{C_0\sum_iQ^{ii}}{u_{11}}-C_0u_{11}-\frac{C_0}{u_{11}}-AC_0,
\end{eqnarray}
where $C_0$ depends on $\|\psi\|_{C^2}$ , $\|u\|_{C^1}$, $\|\underline{u}\|_{C^2}$ and the curvature tensor $R$.

Since $Q^{ii}\geq0$, then by \eqref{req4} and \eqref{Hii},
\begin{eqnarray}\label{uii11}
 \nonumber0 &\geq& \frac{Q^{ii}u_{11ii}}{u_{11}}-\frac{Q^{ii}u_{11i}^2}{u_{11}^2}+aQ^{ii}u_{ii}^2+aQ^{ii}u_ku_{kii}
   +AQ^{ii}(\underline{u}-u)_{ii}\\
   \nonumber&=&\frac{Q^{ii}u_{ii11}}{u_{11}}+\frac{Q^{ii}}{u_{11}}\left(2R^1_{i1i}u_{11}+2R^i_{11i}u_{ii}+R^m_{i1i;i}u_m+
   R^m_{11i;i}u_m\right)\\
   \nonumber&&-\frac{Q^{ii}u_{11i}^2}{u_{11}^2}+aQ^{ii}u_{ii}^2+au_kQ^{ii}u_{kii}+AQ^{ii}(\underline{u}-u)_{ii}\\
   \nonumber&\geq&\frac{Q^{ii}u_{ii11}}{u_{11}}-\frac{Q^{ii}u_{11i}^2}{u_{11}^2}+aQ^{ii}u_{ii}^2+aQ^{ii}u_ku_{kii}
   +AQ^{ii}(\underline{u}-u)_{ii}\\
   &&-C_1\sum_iQ^{ii}-\frac{C_1Q^{ii}|u_{ii}|}{u_{11}}-\frac{C_1\sum_iQ^{ii}}{u_{11}},
\end{eqnarray}
where $C_1$ is a constant depending only on $\|u\|_{C^1}$ and the curvature tensor $R$.

Differentiating \eqref{FU} twice, we get
\begin{eqnarray}\label{Ff}
 \nonumber Q^{ij,rs}u_{ij1}u_{rs1}+Q^{ii}u_{ii11} &=& \psi_{11}+2\psi_{1z}u_1+2\psi_{1p_1}u_{11}+\psi_{zz}u_1^2+2\psi_{zp_1}u_1u_{11} \\
  && +\psi_zu_{11}+\psi_{p_1p_1}u_{11}^2+\psi_{p_i}u_{i11}.
\end{eqnarray}
Note that
\begin{equation}\label{ijrs}
  -Q^{ij,rs}u_{ij1}u_{rs1}\geq-2\sum_{i\geq2}Q^{1i,i1}u_{1i1}^2.
\end{equation}
By \eqref{Hi}, \eqref{Ff} and \eqref{ijrs},
\begin{eqnarray}\label{ij1}
  \nonumber \frac{Q^{ii}u_{ii11}}{u_{11}}
  &\geq&-\frac{1}{u_{11}}Q^{ij,rs}u_{ij1}u_{rs1}+\psi_{p_i}\frac{u_{i11}}{u_{11}}\\
  \nonumber&\geq&-\frac{2}{u_{11}}\sum_{i\geq2}Q^{1i,i1}u_{1i1}^2+\psi_{p_i}\left(-au_iu_{ii}-A\underline{u}_i+Au_i+
  \frac{R^l_{1i1}u_l}{u_{11}}\right)\\
  \nonumber&&-Cu_{11}-\frac{C}{u_{11}}-C\\
  &\geq&-\frac{2}{u_{11}}\sum_{i\geq2}Q^{1i,i1}u_{1i1}^2-a\psi_{p_i}u_iu_{ii}-C_2u_{11}-\frac{C_2}{u_{11}}-AC_2,
\end{eqnarray}
where $C_2$ is a constant depending only on $\|u\|_{C^1}$ , $\|\underline{u}\|_{C^1}$, $\|\psi\|_{C^2}$ and the curvature tensor $R$.

Using \eqref{req1} and \eqref{Quiik}, we have
\begin{eqnarray}\label{aQii}
 \nonumber aQ^{ii}u_ku_{kii}&=&au_kQ^{ii}\left(u_{iik}+R^l_{iki}u_l\right)\\
  \nonumber&=&au_k(\psi_k+\psi_zu_k+\psi_{p_k}u_{kk})+aR^l_{iki}Q^{ii}u_ku_l\\
   &\geq&a\psi_{p_k}u_ku_{kk}-C_3\sum_iQ^{ii}-C_3,
\end{eqnarray}
where $C_3$ is a constant depending only on $\|\psi\|_{C^1}$ , $\|u\|_{C^1}$ and the curvature tensor $R$.
Then \eqref{ht-c2-1} can be derived by \eqref{uii11}, \eqref{ij1} and \eqref{aQii}.

\textbf{Step 2}:
There exists a positive constant $\delta<\frac{1}{n-2}$ such that
$$\frac{C_{n-1}^{k-1} (\tau-\tau(n-2)\delta)^{k-1} +(\tau-1-\tau(n-1)\delta)C_{n-1}^{k-2} (\tau+\tau(n-2)\delta)^{k-2} }{C_n^l (\tau+\tau(n-2)\delta)^l } >\frac{C_{n-1}^{k-1}}{2C_n^l}.$$
We shall show that there exists a constant $B_1=\max\left\{1,\frac{\widetilde{R}}{1-\delta(n-2)}, C_0\left(\frac{a\delta^2}{4n}\left(\frac{C_n^k}{C_n^l}\right)^\frac{1}{k-l}\right)^{-1}\right\}$ for given positive constants $\widetilde{R},\theta,\xi$ such that
$$\frac{a}{4}Q^{ii}u_{ii}^2+\frac{A}{2}Q^{ii}\left(\underline{u}-u\right)_{ii}\geq C_0u_{11},$$
if $u_{11}\geq B_1>1$ and
\begin{equation}\label{A1}
   A=\|\psi\|_{C^0}^{k-l-1}\frac{4k(\tau n-1)C_n^lC_0}{\theta(n-k+l)C_{n-1}^{k-1}}+\frac{4(\tau n-1)}{\theta}\left(\frac{6C_0^4}{1-\xi}+2C_0+\frac{C_0^2}{2a}\right).
\end{equation}

Case 1: $|u_{ii}|\leq \delta u_{11}$ for all $i\geq 2$.\\
In this case we have
$$\left(\tau-1-\tau(n-1)\delta\right)u_{11}\leq U_{11}\leq \left(\tau-1+\tau(n-1)\delta\right)u_{11},$$ $$\left(\tau-\tau(n-2)\delta\right)u_{11}\leq U_{22}\leq \cdots \leq U_{nn}\leq \left(\tau+\tau(n-2)\delta\right)u_{11}.$$
By Theorem 2.18 in \cite{Guan12}, there exist positive constants $\widetilde{R},\theta$ such that
$$F^{ii}(\underline{U}-U)_{ii}\geq\theta(1+\sum_iF^{ii}),$$
when $|\lambda[U]|\geq\widetilde{R}$. Hence, if $u_{11}\geq B_1\geq\frac{\widetilde{R}}{1-\delta(n-2)}$, then
$$\frac{A}{2}Q^{ii}(\underline{u}-u)_{ii}=\frac{A}{2}F^{ii}(\underline{U}-U)_{ii}\geq\frac{A\theta}{2}
\left(1+\sum_iF^{ii}\right)=\frac{A\theta}{2}
\left(1+\frac{1}{\tau n-1}\sum_iQ^{ii}\right).$$
By the definition of $Q^{ii}$, we obtain
\begin{eqnarray*}
  \nonumber\sum_iQ^{ii}&=&(\tau n-1)\sum_iF^{ii} \\
 \nonumber &\geq&\frac{1}{k-l}\left(\frac{\sigma_k}{\sigma_l}\right)^{\frac{1}{k-l}-1}
  \frac{(n-k+l)\sigma_{k-1}\sigma_l-(n-l+1)\sigma_k\sigma_{l-1}}{\sigma_l^2}\\
  \nonumber&\geq&\left(\frac{\sigma_k}{\sigma_l}\right)^{\frac{1}{k-l}-1}
  \frac{\sigma_{k-1}/C_n^{k-1}}{\sigma_l/C_n^k}\\
  \nonumber&=&\frac{C_n^k}{C_n^{k-1}}\left(\frac{\sigma_k}{\sigma_l}\right)^{\frac{1}{k-l}-1}
  \frac{\sigma_{k-1}(U|1)+U_{11}\sigma_{k-2}(U|1)}{\sigma_l}\\
 \nonumber &\geq&\frac{C_n^k}{C_n^{k-1}}\psi^{1-k+l}\frac{C_{n-1}^{k-1}\left(\tau-\tau(n-2)\delta\right)^{k-1}+\left(\tau-1
  -\tau(n-1)\delta\right)C_{n-1}^{k-2}\left(\tau+\tau(n-2)\delta\right)^{k-2}}{C_n^l\left(\tau+\tau(n-2)
  \delta\right)^l}u_{11}\\
  &\geq&\psi^{1-k+l}\frac{(n-k+1)C_{n-1}^{k-1}}{2kC_n^l}u_{11},
\end{eqnarray*}
which implies that
$$\frac{A}{2}Q^{ii}(\underline{u}-u)_{ii}\geq C_0u_{11}.$$

Case 2: $u_{22} > \delta u_{11}$ or $u_{nn} <- \delta u_{11}$.\\
In this case, we have
 \begin{equation*}
\begin{aligned}
\frac{a Q^{ii} u_{ii}^2}{4}&\geq \frac{a}{4} \left(Q^{22} u_{22}^2+Q^{nn} u_{nn}^2\right)
\geq  \frac{a\delta^2}{4} Q^{22} u_{11}^2\\
&\geq  \frac{a\delta^2}{4n} \sum_i F^{ii}u^2_{11}\geq  \left(\frac{C_n^k}{C_n^l}\right)^{\frac{1}{k-l}} \frac{a\delta^2 u_{11}}{4n}u_{11}. \\
\end{aligned}
\end{equation*}
Then, we have
$$\frac{a}{4} Q^{ii} u_{ii}^2\geq C_0 u_{11},$$
if
 $$u_{11} \geq \left(\left(\frac{C_n^k}{C_n^l}\right)^{\frac{1}{k-l}} \frac{a\delta^2}{4n} \right)^{-1}C_0.$$

\textbf{Step 3}:   We show that
$$|u_{ii}|\leq C_4 A,\quad \forall~i\geq2,$$
 if  $u_{11} \geq B_1>1$, where $C_4$ is a constant depending on $n,k,l$, $\|\psi\|_{C^2}$ , $\|u\|_{C^1}$ and the curvature tensor $R$.

Combining with Step 1 and Step 2, we obtain
\begin{eqnarray}\label{ht-c2-32}
\nonumber0&\geq& - \frac{2}{u_{11}} \sum_{i\geq 2} Q^{1i, i1} u_{1i1}^2 -\frac{Q^{ii}u_{11i}^2}{u_{11}^2}
+\frac{aQ^{ii} u_{ii}^2}{4}+\frac{A}{2}Q^{ii}(\underline{u}-u)_{ii}\\
&&-C_0\sum_iQ^{ii}-\frac{C_0^2\sum_iQ^{ii}}{2au_{11}^2}-\frac{C_0}{u_{11}}\sum_iQ^{ii}-C_0\frac{1}{u_{11}}-AC_0.
\end{eqnarray}

Using \eqref{Hi} and Cauchy-Schwarz inequality, we have
\begin{equation}\label{cau}
  \frac{u_{11i}^2}{u_{11}^2}=(au_iu_{ii}+A(\underline{u}-u)_i)^2\leq2a^2u_i^2u_{ii}^2+2A^2(\underline{u}-u)_i^2.
\end{equation}
 By the concavity of $F$ and the definition of $Q^{ii}$,
 \begin{equation}\label{FQ}
   \sum_{i\geq2}Q^{1i,i1}=\sum_{i\geq2}F^{1i,i1}\leq0.
 \end{equation}
 Choosing $a\leq\min\{\frac{1}{64\sup|\nabla u|^2},1\}$. \eqref{ht-c2-32}-\eqref{FQ} imply that
 \begin{eqnarray}\label{a2}
   \nonumber0 &\geq& \left(\frac{a}{4}-2a^2u_i^2\right)Q^{ii}u_{ii}^2-2A^2Q^{ii}(\underline{u}-u)_i^2-2C_0\sum_iQ^{ii}\\
   \nonumber&&-\frac{C_0^2}{2a}\sum_iQ^{ii}-\frac{AC_0}{2}\sum_iQ^{ii}-\frac{AC_0}{2}-(A+1)C_0\\
   &\geq&\frac{a}{8}Q^{ii}u_{ii}^2-\left(2C_0+\frac{C_0^2}{2a}+\frac{AC_0}{2}\right)\sum_iQ^{ii}
   -\left(2A^2+\frac{3A}{2}+1\right)C_0,
 \end{eqnarray}
 if $u_{11}\geq B_1>1$.

 Note that
 $$Q^{ii}\geq Q^{22}\geq\frac{1}{n}\sum_iF^{ii}=\frac{1}{n(\tau n-1)}\sum_iQ^{ii}\geq\frac{1}{n}\left(\frac{C_n^k}{C_n^l}\right)^{\frac{1}{k-l}},\quad\forall~i\geq2.$$
 Thus \eqref{a2} gives that
 $$\frac{a}{8n(\tau n-1)}\left(\sum_{k\geq2}u_{kk}^2\right)\sum_iQ^{ii}\geq\left(\left(2C_0+\frac{C_0^2}{2a}+\frac{AC_0}{2}\right)
 +\frac{C_0\left(2A^2+\frac{3A}{2}+1\right)}{\tau n-1}\left(\frac{C_n^k}{C_n^l}\right)^{-\frac{1}{k-l}}\right)\sum_iQ^{ii},$$
 which implies
 $$\sum_{k\geq 2} u_{kk}^2 \leq C_4^2 A^2.$$

 \textbf{Step 4}:
We shall show that there exists a constant $C$ depending on $n, k, l$, $\|u\|_{C^1}$, $\|\underline{u}\|_{C^2}$, $\|\psi\|_{C^2}$ and the curvature tensor $R$ such that
$$u_{11}\leq C.$$
Without loss of generality, we assume that
\begin{equation}\label{ab}
  u_{11} \geq \max \left\{B, \left(\frac{64A^2|\nabla\underline{u}-\nabla u|^2}{a}\right)^{\frac{1}{2}}, \frac{C_4 A}{\xi}\right\},
\end{equation}
where $B=\max\left\{B_1,(n-2)C_4A+\widetilde{R}\right\}$ and $\xi\leq\frac{1}{2}$ is a constant.

By \eqref{Hi}, \eqref{ab} and Cauchy-Schwarz inequality,
\begin{eqnarray}\label{idn}
  \nonumber\frac{Q^{11}u_{111}^2}{u_{11}^2}&=& Q^{11}\left(au_1u_{11}+A(\underline{u}_1-u_1)\right)^2 \\
  \nonumber&\leq&2a^2|\nabla u|^2Q^{11}u_{11}^2+2A^2Q^{11}(\underline{u}_1-u_1)^2\\
  &\leq&\frac{a}{16}Q^{11}u_{11}^2.
\end{eqnarray}
Combining with Step 3 and \eqref{ab}, we know that $|u_{ii}|\leq \xi u_{11}$ for any $i\geq 2$.
Thus
\begin{equation}\label{beta}
  \frac{1-\xi}{u_{11}-u_{ii}}\leq\frac{1}{u_{11}}\leq \frac{1+\xi}{u_{11}-u_{ii}}.
\end{equation}
By \eqref{beta} and Proposition \ref{th-lem-07}, we obtain
\begin{eqnarray}\label{ht-c2-42}
\nonumber\sum_{i\geq 2} \frac{Q^{ii}u_{11i}^2}{u_{11}^2}&=&\sum_{i\geq 2} \frac{Q^{ii}-Q^{11}}{u_{11}^2}u_{11i}^2 +\sum_{i\geq 2} \frac{Q^{11}u_{11i}^2}{u_{11}^2}\\
\nonumber&\leq& \frac{1+\xi}{u_{11}} \sum_{i\geq 2} \frac{Q^{ii}-Q^{11}}{u_{11}-u_{ii}}u_{11i}^2+\sum_{i\geq 2} \frac{Q^{11}u_{11i}^2}{u_{11}^2}\\
\nonumber&=&\frac{1+\xi}{u_{11}} \sum_{i\geq 2} \frac{F^{11}-F^{ii}}{U_{ii}-U_{11}}u_{11i}^2+\sum_{i\geq 2} \frac{Q^{11}u_{11i}^2}{u_{11}^2}\\
\nonumber&\leq&-\frac{3}{2u_{11}} \sum_{i\geq 2}F^{1i, i1}u_{11i}^2 +\sum_{i\geq 2} \frac{Q^{11}u_{11i}^2}{u_{11}^2}\\
&=&-\frac{3}{2u_{11}} \sum_{i\geq 2}Q^{1i, i1}u_{11i}^2 +\sum_{i\geq 2} \frac{Q^{11}u_{11i}^2}{u_{11}^2},
\end{eqnarray}
the last equality comes from the fact $Q^{1i,i1}=F^{1i,i1}$ for any $i\geq2$.

Using \eqref{Hi}, \eqref{ab} and Cauchy-Schwarz inequality, we get
\begin{eqnarray}\label{q11}
  \nonumber\sum_{i\geq 2} \frac{Q^{11}u_{11i}^2}{u_{11}^2}&\leq&
  2\sum_{i\geq 2}a^2Q^{11}u_i^2u_{ii}^2+2\sum_{i\geq2}A^2Q^{11}(\underline{u}_i-u_i)^2 \\
  \nonumber&\leq& 2a^2\xi^2|\nabla u|^2Q^{11}u_{11}^2+2A^2Q^{11}|\nabla\underline{u}-\nabla u|^2\\
   &\leq&\frac{a}{16}Q^{11}u_{11}^2.
\end{eqnarray}
By Cauchy-Schwarz inequality and Ricci identity, we have
\begin{eqnarray}\label{CSR}
 \nonumber -\frac{2}{u_{11}}\sum_{i\geq2}Q^{1i,i1}u_{1i1}^2 &=& -\frac{2}{u_{11}}\sum_{i\geq2}Q^{1i,i1}(u_{11i}+R^l_{1i1}u_l)^2\\
   &\geq& -\frac{3}{2u_{11}}\sum_{i\geq2}Q^{1i,i1}u_{11i}^2+\frac{6}{u_{11}}\sum_{i\geq2}Q^{1i,i1}(R^l_{1i1}u_l)^2.
\end{eqnarray}
Then \eqref{FQ}, \eqref{beta}-\eqref{CSR} and Proposition \ref{th-lem-07} imply that
\begin{eqnarray}\label{sum}
  \nonumber\sum_{i\geq 2} \frac{Q^{ii}u_{11i}^2}{u_{11}^2} &\leq& -\frac{3}{2u_{11}}\sum_{i\geq2}Q^{1i,i1}u_{11i}^2 +\frac{a}{16}Q^{11}u_{11}^2 \\
  \nonumber &\leq& -\frac{2}{u_{11}}\sum_{i\geq2}Q^{1i,i1}u_{1i1}^2-\frac{6}{u_{11}}\sum_{i\geq2}Q^{1i,i1}(R^l_{1i1}u_l)^2
   +\frac{a}{16}Q^{11}u_{11}^2\\
   \nonumber&\leq&-\frac{2}{u_{11}}\sum_{i\geq2}Q^{1i,i1}u_{1i1}^2+6C_0\sum_{i\geq2}\frac{Q^{ii}-Q^{11}}{u_{11}-u_{ii}}
   +\frac{a}{16}Q^{11}u_{11}^2\\
   &\leq&-\frac{2}{u_{11}}\sum_{i\geq2}Q^{1i,i1}u_{1i1}^2+\frac{6C_0}{1-\xi}\sum_{i\geq2}(Q^{ii}-Q^{11})
   +\frac{a}{16}Q^{11}u_{11}^2,
\end{eqnarray}
if $u_{11}\geq B>1$. Note that
$$\sum_{i\geq2}(Q^{ii}-Q^{11})=\sum_iQ^{ii}-nQ^{11}\leq\sum_iQ^{ii},$$
then substituting \eqref{idn} and \eqref{sum} into \eqref{ht-c2-32}, we derive
\begin{eqnarray*}
 \nonumber 0 &\geq& -\frac{6C_0^4}{1-\xi}\sum_iQ^{ii}+\frac{aQ^{ii}u_{ii}^2}{8}
+\frac{A}{2}Q^{ii}(\underline{u}_{ii}-u_{ii})
-C_0(A+1)-2C_0\sum_iQ^{ii}-\frac{C_0^2}{2a}\sum_iQ^{ii}\\
 \nonumber &\geq&\frac{A}{4}Q^{ii}(\underline{u}_{ii}-u_{ii})-\frac{6C_0^4}{1-\xi}\sum_iQ^{ii}-2C_0\sum_iQ^{ii}
  -\frac{C_0^2}{2a}\sum_iQ^{ii}+\frac{C_0}{2}u_{11}-C_0(A+1)\\
  &\geq&\frac{C_0}{2}u_{11}-C_0(A+1),
\end{eqnarray*}
if $u_{11}\geq B\geq(n-2)C_4A+\widetilde{R}$ and $A$ defined as \eqref{A1}.

It follows that
$$u_{11}\leq 2(A+1).$$
\end{proof}

Now we consider the estimates for the second order derivatives on the
boundary $\partial M$. For any fixed $x_0\in\partial M$, we can choose smooth orthonormal local frames $e_1, \cdots,e_n$ around
$x_0$ such that when restricted on $\partial M$, $e_n$ is normal to $\partial M$. For $x\in\overline{M}$,
let $\rho(x)$ and $d(x)$ denote the distances from $x$ to $x_0$ and $\partial M$ respectively,
$$\rho(x)=dist_{M}(x,x_0),\quad d(x)=dist_{M}(x,\partial M),$$
and $M_{\delta}=\{x\in M:\rho(x)<\delta\}$.
Since $\nabla_{ij}\rho^2(x_0) = 2\delta_{ij}$, we may assume $\rho$ is smooth in $M_{\delta_0}$ for fixed $\delta_0>0$ and
$$I\leq\nabla_{ij}\rho^2\leq3I\quad in~M_{\delta_0}.$$

Then we get the following important lemma, which plays key role in
our boundary estimates.
\begin{lemma}\label{LQ}
 Let
 $$L=Q^{ij}\nabla_{ij}-\psi_{p_i}\nabla_i, \quad  v=u-\underline{u}+td-\frac{N}{2}d^2,$$
 then for a positive constant $\varepsilon_0$, there exist some uniform  constants $t, \delta$ sufficiently small and $N$ sufficiently large such that
 \begin{equation*}
\left\{
\begin{aligned}
&Lv\leq-\frac{\varepsilon_0}{4}(1+\sum_i F^{ii}) &&in~
M_{\delta},\\
&v\geq0 &&on~\partial M_{\delta}.
\end{aligned}
\right.
\end{equation*}
\end{lemma}
\begin{proof}
It is easy to see that  $v(x)=0$ for any $x\in\partial M\bigcap B_{\delta}$. Then we can choose $\delta<\frac{2t}{N}$ such that $v(x)\geq0$ for any $x\in M\cap \partial B_{\delta}$. Therefore
$$v\geq0 \quad \mbox{on}~\partial M_{\delta}.$$
 Let $\mu=\lambda[\underline{U}]$ and $\lambda=\lambda[U]$ be the eigenvalues of $\underline{U}$ and $U$ respectively.
As the result in \cite{Guan14}, denote $\nu_{\chi}:=\frac{D f(\chi)}{|Df(\chi)|}$ to be the unit normal vector to the level hypersurface $\partial \Gamma^{f(\chi)}$ for $\chi \in \Gamma$, $\Gamma$ is a symmetric open and convex cone in $\mathbb{R}^n$ with $\Gamma_n\subset\Gamma$.
Note that $\{\mu(x)\mid x\in \overline{M}\}$ is a compact subset of $\Gamma$, there exists a uniform constant $\beta\in(0, \frac{1}{2\sqrt{n}})$ such that
$$\nu_{\mu(x)} -2 \beta \mathbf{1} \in \Gamma_n, \quad \forall x\in \overline{M}.$$
We divide into two cases to estimate $Lv$.

Case 1: $|\nu_{\mu}-\nu_{\lambda}|<\beta$.

Since $\nabla_n d(x_0)=1$, $\nabla_{\alpha}d(x_0)=0$ for all $\alpha <n$, we can choose a
constant $\delta_0$ such that
\begin{eqnarray*}
\frac{1}{2}\leq |\nabla d|\leq 1, \quad  -\widetilde{C}_{1}I\leq \nabla^2d \leq \widetilde{C}_{1}I, \quad \forall x\in M_{\delta}
\end{eqnarray*}
for any $\delta < \delta_0$, where $\widetilde{C}_{1}$ depends on the geometry of $\partial M$.
Note that $\nu_{\lambda}-\beta\mathbf{1}\in\Gamma_n$, $\mathbf{1}=(1,\cdots,1)$, then
\begin{eqnarray}\label{DHQ-lem-for-1}
F^{ii}\geq\frac{\beta}{\sqrt{n}}\sum_kF^{kk},\quad \forall~1\leq i\leq n.
\end{eqnarray}
By the definition of $v$, we have
\begin{eqnarray}\label{DHQ-lem-for-11}
 \nonumber Lv &=& Q^{ij}\nabla_{ij}v-\psi_{p_i}\nabla_i v \\
  \nonumber &=&Q^{ij}\left(\nabla_{ij}(u-\underline{u})+t \nabla_{ij}d - N \nabla_id \nabla_jd -Nd \nabla_{ij}d\right)\\
  \nonumber&&-\psi_{p_i}\left(\nabla_i(u-\underline{u}) +t \nabla_id - N d\nabla_id \right)\\
  &\leq&Q^{ij}\nabla_{ij}(u-\underline{u})+ (t-Nd)Q^{ij}\nabla_{ij}d-N Q^{ij}\nabla_id \nabla_jd+\widetilde{C}_{2}+\widetilde{C}_{2}t+\widetilde{C}_{2}N \delta,
\end{eqnarray}
where $\widetilde{C}_{2}$ depends on $\|\psi\|_{C^1}$, $\|u\|_{C^1}$ and $\|\underline{u}\|_{C^1}$.
By the concavity of $F$, we have
$$Q^{ij}\nabla_{ij}(u-\underline{u})=F^{ij} (\nabla_{ij}U- \nabla_{ij}\underline{U})\leq 0.$$
Thus, we have
\begin{eqnarray}\label{eq1}
Lv&\leq& (t-Nd)Q^{ij}\nabla_{ij}d-N Q^{ij}\nabla_id \nabla_jd+\widetilde{C}_{2}+\widetilde{C}_{2}t+\widetilde{C}_{2}N \delta\\
\nonumber&\leq& \widetilde{C}_{1} t\sum_i Q^{ii}+\widetilde{C}_{2}+\widetilde{C}_{2} t+ N \widetilde{C}_{1}\delta \sum_i Q^{ii}-N Q^{ij}\nabla_id \nabla_jd+\widetilde{C}_{2}N \delta.
\end{eqnarray}
By \eqref{DHQ-lem-for-1}, we have
\begin{eqnarray}\label{DHQ-lem-for-2}
Q^{ij}\nabla_id \nabla_jd=\tau \sum_l F^{ll} |\nabla d|^2- F^{ij} \nabla_id \nabla_jd\geq \frac{(n-1)\beta}{4\sqrt{n}} \sum_l F^{ll}.
\end{eqnarray}
Note that
\begin{eqnarray}\label{DHQ-lem-for-3}
\sum_l F^{ll} =\frac{1}{\tau n-1}\sum_iQ^{ii}\geq \left(\frac{C_n^k}{C_n^l}\right)^{\frac{1}{k-l}}:=\widetilde{C}_{3}.
\end{eqnarray}
Combining with \eqref{eq1}-\eqref{DHQ-lem-for-3}, we get
\begin{eqnarray*}
\nonumber Lv&\leq& \widetilde{C}_{1} t(\tau n-1) \sum_l F^{ll}+\frac{\widetilde{C}_{2}+\widetilde{C}_{2} t}{\widetilde{C}_{3}} \sum_l F^{ll}\\
\nonumber&&+ N \widetilde{C}_{1}(\tau n-1)\delta \sum_l F^{ll}-N \frac{(n-1)\beta}{4\sqrt{n}} \sum_l F^{ll}+\frac{\widetilde{C}_{2}N \delta}{\widetilde{C}_{3}}\sum_l F^{ll}\\
&\leq & -\frac{\widetilde{C}_{2}}{\widetilde{C}_{3}} \sum_{l} F^{ll},
\end{eqnarray*}
if we choose the constants $t, N, \delta$ satisfying
\begin{equation*}
\left\{
\begin{aligned}
&t\leq \frac{\widetilde{C}_{2}}{\widetilde{C}_{1} \widetilde{C}_{3}(\tau n-1)+\widetilde{C}_{2}},\\
&N\geq \frac{20\widetilde{C}_{2} \sqrt{n}}{\widetilde{C}_{3} (n-1) \beta},\\
&\delta\leq \min\{\delta_0, \frac{2t}{N}\}.
\end{aligned}
\right.
\end{equation*}

Case 2: $|\nu_{\mu}-\nu_{\lambda}|\geq\beta$.

From Lemma 2.1 in \cite{Guan14}, we know that for some uniform constant $\varepsilon_0>0$,
$$Q^{ij}\nabla_{ij}(\underline{u}-u)=F^{ij}\nabla_{ij}(\underline{U}-U)\geq f_i(\mu_i-\lambda_i)\geq\varepsilon_0(1+\sum_iF^{ii}).$$
In according to \eqref{DHQ-lem-for-11}, we have
\begin{eqnarray}\label{w1}
  \nonumber Lv &\leq&-\frac{\varepsilon_0}{2}(1+\sum_iF^{ii})- \frac{1}{2} \left( Q^{ij}\nabla_{ij}(\underline{u}-u)+2N Q^{ij}\nabla_id \nabla_jd\right)\\
  &&+ (t-Nd)\widetilde{C}_{1}(\tau n-1) \sum_l F^{ll}+\widetilde{C}_{2}+\widetilde{C}_{2}t+\widetilde{C}_{2}N \delta.
\end{eqnarray}
By the concavity of $F$,
\begin{eqnarray}\label{w2}
  \nonumber Q^{ij}\nabla_{ij}(\underline{u}-u)+2NQ^{ij}\nabla_id\nabla_jd &=& F^{ij}\nabla_{ij}(\underline{U}-U)+2NF^{ij}(\tau|Dd|^2\delta_{ij}-\nabla_id\nabla_jd) \\
 \nonumber &\geq&F\left(\nabla_{ij}\underline{U}+2N(\tau|Dd|^2\delta_{ij}-\nabla_id\nabla_jd)\right)-F(\nabla_{ij}U)\\
 \nonumber &\geq& \left(\frac{\sigma_k}{\sigma_l}\right)^{\frac{1}{k-l}}(\mu+2N\lambda[A])-\psi,
\end{eqnarray}
where $\mu=(\mu_1,\mu_2,\cdots,\mu_n)$ and $A=\left[\tau|Dd|^2\delta_{ij}-d_id_j\right]_{n\times n}$. Since  $\lambda[A]\geq \frac{1}{4}\mbox{diag}(0,1,\cdots,1)$, then we have
\begin{eqnarray}\label{w2}
  \nonumber Q^{ij}\nabla_{ij}(\underline{u}-u)+2NQ^{ij}\nabla_id\nabla_jd \geq \left(\frac{\sigma_k}{\sigma_l}\right)^{\frac{1}{k-l}}(\mu+ \overline{\lambda})-\widetilde{C}_{2},
\end{eqnarray}
where $\overline{\lambda}= \mbox{diag}(0, \frac{N}{2}, \cdots, \frac{N}{2})$.

Next, based on the range of $k$, we consider the following two conditions.

When $k=n$, since $F(\underline{U})\geq \psi(x, \underline{u}, D\underline{u})>0$, we know that $\sigma_n(\mu)\geq \widetilde{C}_{4}$, then
\begin{equation}\label{w3}
  \frac{\sigma_n}{\sigma_l}(\mu+ \overline{\lambda})\geq
   \frac{\widetilde{C}_{4} N^{n-1}}{C_n^l(\mu_{\mbox{max}}+N)^l}\geq \widetilde{C}_{5} N^{n-1-l},
\end{equation}
where $\widetilde{C}_{5}$ depends on $\inf \psi, n, k, l$ and $\|\underline{u}\|_{C^2}$.

When $3\leq k\leq n-1$, we assume that $N>|4\mu|^2+1$, then
\begin{eqnarray}\label{w4}
  \frac{\sigma_k}{\sigma_l}(\mu+ \overline{\lambda})&\geq& \frac{\sigma_k}{\sigma_l} \left(\mbox{diag}(\mu_1, \frac{N}{4}, \cdots, \frac{N}{4})\right)\\
  \nonumber &=& \frac{C_{n-1}^k (\frac{N}{4})^{k}+\mu_1 C_{n-1}^{k-1} (\frac{N}{4})^{k-1}}{C_{n-1}^l (\frac{N}{4})^{l}+\mu_1 C_{n-1}^{l-1} (\frac{N}{4})^{l-1}}\\
  \nonumber&\geq& \frac{C_{n-1}^k (\frac{N}{4})^{k}- \frac{N^{\frac{1}{2}}}{4} C_{n-1}^{k-1} (\frac{N}{4})^{k-1}}{C_{n-1}^l (\frac{N}{4})^{l}+C_{n-1}^{l-1} (\frac{N}{4})^{l}}\\
  \nonumber&\geq& \widetilde{C}_{5} N^{k-l-1}.
\end{eqnarray}
By \eqref{w1}-\eqref{w4}, we can choose $t, N, \delta$ satisfying
\begin{equation*}
\left\{
\begin{aligned}
&t\leq \min \left\{\frac{\varepsilon_0}{12 \widetilde{C}_{1}(\tau n-1)}, 1\right\},\\
&N\geq \max\left\{ \left(10\widetilde{C}_{2}\widetilde{C}_{5}^{-\frac{1}{k-l}}\right)^{\frac{k-l}{k-l-1}}, 16n\mu_{\mbox{max}}^2+1\right\},\\
&\delta\leq \min\left\{\delta_0, \frac{2t}{N}\right\}.
\end{aligned}
\right.
\end{equation*}
Thus
\begin{eqnarray*}
  Lv &\leq&-\frac{\varepsilon_0}{2}(1+\sum_iF^{ii})-\frac{1}{2}\left(\widetilde{C}_{5}N^{k-l-1}\right)^{\frac{1}{k-l}}+ 3t\widetilde{C}_{1}(\tau n-1) \sum_l F^{ll}+5\widetilde{C}_{2}\\
  &\leq&-\frac{\varepsilon_0}{4}(1+\sum_iF^{ii}) .
\end{eqnarray*}
\end{proof}

\begin{theorem}\label{C2-1}
Let $u\in C^{\infty}(M)$ be an admissible solution for equation \eqref{Eq}. Under the  assumptions mentioned  in Theorem \ref{main},
there exists a constant $C$ depending only on $n, k, l, \|u\|_{C^1}, \|\underline{u}\|_{C^2}, \inf \psi$ , $\|\psi\|_{C^2}$ and the curvature tensor $R$ such that
$$\sup_{ \overline{M}} |\nabla^2 u | \leq C.$$
\end{theorem}
\begin{proof}
By Theorem \ref{C2-0}, we only need to derive boundary estimates.
For any $ x_0\in\partial M$, we can choose the local frames $e_1,\cdots,e_n$ around $x_0$ such
that $e_n$ is interior normal to $\partial M$.

$\mathbf{Case~1:} $  Estimates of $\nabla_{\alpha\beta}u, \alpha, \beta=1,\cdots,n-1$ on $\partial M$.

Consider a point $x_0\in\partial M$. Since $u-\underline{u}=0$ on $\partial M$, therefore,
$$\nabla_{\alpha\beta}(u-\underline{u})=-\nabla_{n}(u-\underline{u})B_{\alpha\beta} \quad \mbox{on} ~\partial M,$$
where $B_{\alpha\beta} = \langle\nabla_{\alpha}e_{\beta},e_n\rangle$ denotes the second fundamental form of $\partial M$. Therefore,
\begin{eqnarray*}
|\nabla_{\alpha\beta}u| \leq C \quad \mbox{on} ~\partial M,
\end{eqnarray*}
where $C$ depends on $\|u\|_{C^1}$ and $\|\underline{u}\|_{C^2}$.

$\mathbf{Case~2:}$ Estimates of $ \nabla_{\alpha n}u$, $\alpha=1,\cdots,n-1$ on $\partial M$.\\
Let
\begin{equation}\label{Phi}
  \Phi=A_1v+A_2\rho^2-A_3\sum_{\beta<n}|\nabla_{\beta}(u-\underline{u})|^2,
\end{equation}
then combining with Lemma \ref{LQ}, we claim that
\begin{equation*}
\left\{
  \begin{aligned}
     & L(\Phi\pm\nabla_{\alpha}(u-\underline{u}))\leq0 &&in~M_{\delta},\\
     & \Phi\pm\nabla_{\alpha}(u-\underline{u})\geq0 &&on~\partial M_{\delta},
   \end{aligned}
   \right.
\end{equation*}
for suitably chosen positive constants $A_1, A_2, A_3$ and $L$ , $v$ are defined in Lemma \ref{LQ}. First we have for some uniform constant $\widehat{C}_0$,
\begin{equation*}
  L(\rho^2)=Q^{ij}\nabla_{ij}(\rho^2)-\psi_{p_i}\nabla_i(\rho^2)\leq \widehat{C}_0(1+\sum_iQ^{ii}),
\end{equation*}
and by \eqref{req0}, \eqref{Quii} and \eqref{Quiik},
\begin{eqnarray}\label{L2}
 \nonumber |L\nabla_{\alpha}(u-\underline{u})|&\leq&2Q^{ij}\Gamma_{i\alpha}^l\nabla_{jl}u+C(1+\sum_iQ^{ii})\\
  &\leq& \widehat{C}_1(1+\sum_i\widetilde{f}_i|\widetilde{\lambda}_i|+\sum_i\widetilde{f}_i),
\end{eqnarray}
where $\widetilde{\lambda}_i (i=1, \cdots, n)$ are the  eigenvalues of $\nabla^2u$.

Furthermore, we get
\begin{eqnarray}\label{nab}
\nonumber L|\nabla_{\beta}(u-\underline{u})|^2 &=& 2Q^{ij}\nabla_{\beta}(u-\underline{u})\nabla_{ij}\nabla_{\beta}(u-\underline{u})+2Q^{ij}\nabla_i\nabla_{\beta}(u-\underline{u}) \nabla_{j}\nabla_{\beta}(u-\underline{u})\\
  \nonumber&& -2\psi_{p_i}\nabla_{\beta}(u-\underline{u})\nabla_i\nabla_{\beta}(u-\underline{u})\\
  &\geq& 2Q^{ij}u_{i\beta}u_{j\beta}-\widehat{C}_2\left(1+\sum_i\widetilde{f}_i|\widetilde{\lambda}_i|+\sum_i\widetilde{f}_i\right).
\end{eqnarray}
By Proposition 2.19 in \cite{Guan12}, we know that there exists an index $r$ such that
\begin{equation}\label{gueq1}
  \sum_{\beta<n}Q^{ij}u_{i\beta}u_{j\beta}\geq\frac{1}{2}\sum_{i\neq r}\widetilde{f}_i\widetilde{\lambda}_i^2.
\end{equation}
Since $\widetilde{f}$ satisfies $\widetilde{f}_i=\frac{\partial\widetilde{f}}{\partial\widetilde{\lambda}_i}=\sum_iQ^{ii}>0$ , $\sum_i\widetilde{f}_i\widetilde{\lambda}_i=\sum_iQ^{ii}u_{ii}=\psi>0$ and $\widetilde{f}$ is a concave function, then by Corollary 2.21
in \cite{Guan12}, for index $r$ and $\varepsilon>0$,
\begin{equation}\label{gueq2}
  \sum_i\widetilde{f}_i|\widetilde{\lambda}_i|\leq\varepsilon\sum_{i\neq r}\widetilde{f}_i\widetilde{\lambda}_i^2+\frac{C}{\varepsilon}\sum_i\widetilde{f}_i+Q(r),
\end{equation}
where $Q(r)=\widetilde{f}(\widetilde{\lambda})-\widetilde{f}(\mathbf{1})$ if $\lambda_r\geq0$, $\mathbf{1}=(1,\cdots,1)$ and for some constant $K_0\geq0$,
$$Q(r)=\varepsilon nK_0^2\min_{1\leq i\leq n}\frac{1}{\widetilde{f}_i},\quad \mbox{if} ~\lambda_r<0.$$

Hence \eqref{L2}-\eqref{gueq2} yield that
\begin{eqnarray*}
 \nonumber &&A_3\sum_{\beta<n}L|\nabla_{\beta}(u-\underline{u})|^2\pm L(\nabla_{\alpha}(u-\underline{u}))\\
 \nonumber&\geq& 2A_3\sum_{\beta<n}Q^{ij}u_{i\beta}u_{j\beta}
 -\widehat{C}_1\left(1+\sum_i\widetilde{f}_i|\widetilde{\lambda}_i|+\sum_i\widetilde{f}_i\right)\\
 \nonumber&&-A_3\widehat{C}_2(n-1)\left(1+\sum_i\widetilde{f}_i|\widetilde{\lambda}_i|+\sum_i\widetilde{f}_i\right)\\
 \nonumber&\geq&A_3\sum_{i\neq r}\widetilde{f}_i\widetilde{\lambda}_i^2-(A_3\widehat{C}_2(n-1)+\widehat{C}_1)\left(1+\varepsilon\sum_{i\neq r}\widetilde{f}_i\widetilde{\lambda}_i^2+\frac{C}{\varepsilon}\sum_i\widetilde{f}_i+\sum_i\widetilde{f}_i+Q (r)\right)\\
 \nonumber&\geq&\left(A_3-A_3\widehat{C}_2(n-1)\varepsilon-\widehat{C}_1\varepsilon\right)\sum_{i\neq r}\widetilde{f}_i\widetilde{\lambda}_i^2-A_3\widehat{C}_3(1+\sum_i\widetilde{f}_i)\\
 &\geq&-A_3\widehat{C}_3(1+\sum_i\widetilde{f}_i),
\end{eqnarray*}
by choosing $0<\varepsilon<\min\left\{\frac{1}{\widehat{C}_2(n-1)},1\right\}$ and $A_3>\max\left\{\frac{\widehat{C}_1\varepsilon}{1-\widehat{C}_2(n-1)\varepsilon},1\right\}$.
Combine with Lemma \ref{LQ} and choose $A_1\gg A_2\gg A_3\gg1$, then
\begin{equation*}
\left\{
\begin{aligned}
&L\left(\Phi\pm\nabla_{\alpha}(u-\underline{u})\right)\leq0\quad &&in~M_{\delta},\\
&\Phi\pm\nabla_{\alpha}(u-\underline{u})\geq0\quad &&on~\partial M_{\delta}.
\end{aligned}
\right.
\end{equation*}
Therefore by the maximum principle, we have
$$\Phi\pm\nabla_{\alpha}(u-\underline{u})\geq0\quad in ~M_{\delta}.$$
Thus we obtain
$$|\nabla_{n\alpha}u(x_0)|\leq\nabla_n\Phi(x_0)+|\nabla_{n\alpha}\underline{u}(x_0)|\leq C, \quad\alpha=1, \cdots, n-1. $$

$\mathbf{Case~3:} $ Estimates of $\nabla_{nn}u$ on $\partial M$.\\

We only need to show the uniform upper bound
$$\nabla_{nn}u(x_0)\leq C, \quad \forall~x_0\in\partial M,$$
since $\Gamma_k\subset \Gamma_1$ implies $\Delta u \geq 0$ and the lower bound for $\nabla_{nn} u$ follows from the estimate of $\nabla_{\alpha \beta} u$ and $\nabla_{\alpha n} u$.
We will divide the proof into two conditions. The case $\tau=1$ is more complicated and need classified discussion.

When $\tau>1$,  we have
\begin{eqnarray*}
\nonumber [U(x_0)]&=&\tau \Delta u(x_0) I- \nabla^2u(x_0)\\
\nonumber &\geq& \mbox{diag}(\tau \nabla_{nn}u(x_0), \cdots, \tau \nabla_{nn}u(x_0), (\tau-1) \nabla_{nn}u(x_0)) -\overline{C}_0I\\
\nonumber &\geq& ((\tau-1) \nabla_{nn}u(x_0)-\overline{C}_0)I,
\end{eqnarray*}
where $\overline{C}_0$ depends on $\|\nabla_{\alpha\beta}u\|_{C^0}$ and $\|\nabla_{\alpha n}u\|_{C^0}$. It is clear that
\begin{eqnarray*}
 \nonumber \psi(x_0, u(x_0), \nabla u(x_0))=F(U)(x_0)&=& F^{ij}(x_0) U_{ij}(x_0)\\
 &\geq & ((\tau-1) \nabla_{nn}u(x_0)-\overline{C}_0) \sum_l F^{ll}.
\end{eqnarray*}
Thus we obtain the upper bound as desired.

When $\tau=1$. By lemma 1.2 of \cite{CNS85} and the estimates of $\nabla_{\alpha \beta}u$, $\nabla_{\alpha n}u$, we can choose $R_1>0$ sufficiently large  such that if $\nabla_{nn}u(x_0)> R_1$,
\begin{equation*}
\left\{
\begin{aligned}
\widetilde{\lambda}_i[\nabla_{ij}u(x_0)]&=\widetilde{\lambda}_i^\prime [\nabla_{\alpha \beta}u(x_0)]+ o(1), \quad i=1, \cdots, n-1,\\
\widetilde{\lambda}_n [\nabla_{ij}u(x_0)]&=\nabla_{nn}u(x_0)\left(1+O(\frac{1}{\nabla_{nn}u(x_0)})\right). \end{aligned}
\right.
\end{equation*}
Here $\widetilde{\lambda}[\nabla_{ij}u] =(\widetilde{\lambda}_1[\nabla_{ij}u], \cdots, \widetilde{\lambda}_{n}[\nabla_{ij}u])$ denotes the eigenvalues of the $n\times n$ matrix $\nabla^2u$ and $\widetilde{\lambda}^\prime[\nabla_{\alpha \beta}u] =(\widetilde{\lambda}_1^\prime[\nabla_{\alpha \beta}u], \cdots, \widetilde{\lambda}_{n-1}^\prime[\nabla_{\alpha \beta}u])$ denotes the eigenvalues of the $(n-1)\times (n-1)$ matrix $\left[\nabla_{\alpha\beta}u\right]_{1\leq \alpha, \beta\leq n-1}$. For convenience, we denote
\begin{eqnarray*}
\widetilde{\lambda}_i= \widetilde{\lambda}_i[\nabla_{ij}u], \quad \widetilde{\lambda}_{\alpha}^\prime= \widetilde{\lambda}_{\alpha}^\prime [\nabla_{\alpha \beta}u], \quad \widehat{\lambda}_i= \sum_{l=1}^n \widetilde{\lambda}_{l}-\widetilde{\lambda}_i ,\quad \widehat{\lambda}_{\alpha}^\prime= \sum_{i=1}^{n-1} \widetilde{\lambda}_i^\prime- \widetilde{\lambda}_{\alpha}^\prime.
\end{eqnarray*}

If $k<n$,
\begin{eqnarray*}
 && F^{k-l}(U)(x_0)\\
 \nonumber&=& \frac{\sigma_k}{\sigma_l}\left(\sum_i \widetilde{\lambda}_i(x_0) -\widetilde{\lambda}_1(x_0), \cdots, \sum_i \widetilde{\lambda}_i(x_0) -\widetilde{\lambda}_n(x_0)\right)\\
 \nonumber&= & \frac{\sigma_k}{\sigma_l}\left(\widetilde{\lambda}_n(x_0)+\widehat{\lambda}_{1}^\prime(x_0)+o(1), \cdots, \widetilde{\lambda}_n(x_0)+\widehat{\lambda}_{n-1}^\prime(x_0)+o(1), \sum_{i=1}^{n-1} \widetilde{\lambda}_i^\prime(x_0)  +o(1)\right)\\
 \nonumber&\geq & \frac{\widetilde{\lambda}^{k}_n(x_0)+ o\left(\widetilde{\lambda}^{k-1}_n(x_0)\right)}{C_n^l\widetilde{\lambda}^{l}_n(x_0)+O\left(
 \widetilde{\lambda}^{l-1}_n(x_0)\right)},
\end{eqnarray*}
which implies the uniform upper bound of $\nabla_{nn}u(x_0)$.

If $k=n$, we show the uniform upper bound of $\nabla_{nn}u(x_0)$ by proving that there are uniform constants $\overline{C}_1$ such that
\begin{equation}\label{w89}
\min_{x \in \partial M} tr([\nabla_{\alpha\beta}u])\geq \overline{C}_1>0.
\end{equation}
Suppose we have found such $\overline{C}_1$, then
\begin{eqnarray*}
 && F^{n-l}(U)(x_0)\\
 \nonumber&= & \frac{\sigma_n}{\sigma_l}\left(\widetilde{\lambda}_n(x_0)+\widehat{\lambda}_{1}^\prime(x_0)+o(1), \cdots, \widetilde{\lambda}_n(x_0)+\widehat{\lambda}_{n-1}^\prime(x_0)+o(1), \sum_{i=1}^{n-1} \widetilde{\lambda}_i^\prime(x_0)  +o(1)\right)\\
 \nonumber&\geq & \frac{\overline{C}_1\widetilde{\lambda}^{n-1}_n(x_0)+ o\left(\widetilde{\lambda}^{n-1}_n(x_0)\right)}{C_n^l\widetilde{\lambda}^{l}_n(x_0)+O\left(\widetilde{\lambda}^{l-1}_n(x_0)\right)},
\end{eqnarray*}
which implies the uniform upper bound of $\nabla_{nn}u(x_0)$. Hence we only need to prove \eqref{w89}.

Suppose that $tr([\nabla_{\alpha\beta}u])$ attains its minimum at $x_1 \in \partial M$. To show \eqref{w89}, we may assume $tr([\nabla_{\alpha\beta}u(x_1)])< \frac{1}{2}tr([\nabla_{\alpha\beta}\underline{u}(x_1)])$, since otherwise we are done as $tr([\nabla_{\alpha\beta}\underline{u}(x_1)])=\sum_i \widetilde{\lambda}_i[\nabla_{ij}\underline{u}] -\widetilde{\lambda}_n[\nabla_{ij}\underline{u}]>\overline{C}_2$. Let us compute
\begin{eqnarray*}
 \nabla_{\alpha\alpha}u=\nabla_{\alpha\alpha}\underline{u}-  \nabla_n(u-\underline{u}) B_{\alpha\alpha} \quad \mbox{on}~\partial M.
\end{eqnarray*}
It follows that
\begin{eqnarray}\label{w42}
 \nonumber\nabla_n(u-\underline{u})(x_1) \sum_{\alpha} B_{\alpha\alpha}(x_1)&=&tr([\nabla_{\alpha\beta}\underline{u}(x_1)])-tr([\nabla_{\alpha\beta}u(x_1)])\\
  &\geq& \frac{1}{2}tr([\nabla_{\alpha\beta}\underline{u}(x_1)])>\frac{\overline{C}_2}{2}.
\end{eqnarray}
For any $x\in \partial M$ near $x_0$, applying that $tr([\nabla_{\alpha\beta}u])\mid_{\partial M}$ is minimized at $x_1$ yields,
$$\nabla_n(u-\underline{u})(x) \sum_{\alpha}B_{\alpha\alpha}(x)\leq tr([\nabla_{\alpha\beta}\underline{u}(x)])-tr([\nabla_{\alpha\beta}\underline{u}(x_1)])  +\nabla_n(u-\underline{u})(x_1) \sum_{\alpha}B_{\alpha\alpha}(x_1).$$
Note that $B_{\alpha\alpha}$ is smooth near $\partial M$ and $0<u-\underline{u}\leq C$, adding \eqref{w42}, we can choose a constant $\delta$ sufficiently small such that
$$\sum_{\alpha} B_{\alpha\alpha}  \geq \overline{C}_3>0   \quad \mbox{in}~ M\cap B_{\delta}(x_1)$$
for some uniform constant $\overline{C}_3>0$. Therefore
\begin{eqnarray*}
\nabla_n(u-\underline{u})(x_1) = \Psi(x_1), \quad \nabla_n(u-\underline{u})(x) \leq \Psi(x)\quad \mbox{on} ~B_{\delta}(x_1)\cap \partial M,
\end{eqnarray*}
where $\Psi= \left(\sum_{\alpha} B_{\alpha\alpha} (x) \right)^{-1} \left(tr([\nabla_{\alpha\beta}\underline{u}(x)])-tr([\nabla_{\alpha\beta}\underline{u}(x_1)])  +\nabla_n(u-\underline{u})(x_1) \sum_{\alpha}B_{\alpha\alpha}(x_1)\right)$ is smooth in $M\cap B_{\delta}(x_1)$.

We now apply the argument of $\nabla_{\alpha n}$ again. For $A_1\gg A_2\gg A_3\gg 1$, it remains to prove that
\begin{equation*}
\left\{
\begin{aligned}
&L\left(\Phi+\Psi-\nabla_n(u-\underline{u})\right)\leq0\quad &&in~ M\cap B_{\delta}(x_1),\\
&\Phi+\Psi-\nabla_n(u-\underline{u})\geq 0\quad &&on~ \partial \left(M\cap B_{\delta}(x_1)\right).
\end{aligned}
\right.
\end{equation*}
 According to the maximum principle, we have
$$\Phi+\Psi-\nabla_n(u-\underline{u})\geq 0\quad \mbox{in}~ M\cap B_{\delta}(x_1).$$
Thus $\nabla_n\Psi(x_1)-\nabla_{nn}(u-\underline{u})(x_1)\geq-\nabla_n\Phi(x_1)\geq-C$, which implies that $\nabla_{nn}u(x_1)\leq C$.
Therefore,
$$\widetilde{\lambda}_i(x_1)\leq C, \quad i=1,\cdots,n.$$
It is clear that  $\widehat{\lambda}=(\widehat{\lambda}_1,\cdots,\widehat{\lambda}_n)\in \Gamma_n$. Hence
\begin{eqnarray*}
 \psi^{n-l}(x_1, u(x_1), \nabla u(x_1))&=&\frac{\sigma_n(\widehat{\lambda})}{\sigma_l(\widehat{\lambda})}\leq \frac{\sigma_n(\widehat{\lambda})}{ C_n^l \sigma^{\frac{l}{n}}_n(\widehat{\lambda})}=\frac{\sigma^{1-\frac{l}{n}}_n(\widehat{\lambda})}{C_n^l}.
\end{eqnarray*}
It follows that
$$\widehat{\lambda}_i (x_1) \geq \overline{C}_4.$$
We assume without loss of generality that the eigenvalue  $\widehat{\lambda}_i(x_1)$ satisfy $\widehat{\lambda}_1(x_1)\leq \widehat{\lambda}_2(x_1)\leq \cdots \leq\widehat{\lambda}_n(x_1)$. According to the Cauchy interlacing inequalities (see e.g. \cite{Wil63}, p. 103-104),
$$\widehat{\lambda}_{\alpha} (x_1)\leq \widehat{\lambda}_{\alpha}^\prime(x_1)\leq \widehat{\lambda}_{\alpha+1} (x_1).$$
Hence the claim \eqref{w89} holds and we obtian the upper bound of $\nabla_{nn}u(x_0)$ as desired.
\end{proof}

\begin{proof}[Proof of Theorem \ref{main-1}]
The theorem can be easily obtained by Lemma \ref{C0},  Theorem \ref{C1},   \ref{C2-0} and  \ref{C2-1}.
\end{proof}

\begin{proof}[Proof of Theorem \ref{main}]
From Theorem \ref{main-1}, we have
uniform estimates in $C^2(\overline{M}
)$ for classical elliptic solutions of the Dirichlet problems:
\begin{equation*}
\left\{
\begin{aligned}
&\left(\frac{\sigma_k}{\sigma_l}\right)^{\frac{1}{k-l}}(U)=tf(x,u,\nabla u)+(1-t) \left(\frac{\sigma_k}{\sigma_l}\right)^{\frac{1}{k-l}}(\underline{U})&&in~
M,\\
&u = \varphi &&on~\partial M,
\end{aligned}
\right.
\end{equation*}
for $0\leq t\leq 1$. Theorem \ref{main} then follows from the Evans-Krylov second derivative H$\ddot{o}$lder estimates of
Evans, Krylov and Caffarelli-Nirenberg-Spruck, and the method of continuity, for more details see \cite{Gi98}. The uniqueness assertion is immediate from the maximum principle.
\end{proof}



\end{document}